\documentclass[final,leqno,onefignum,onetabnum]{siamltex1213}
\usepackage{amsfonts}
\usepackage{amssymb}
\usepackage{mathrsfs}
\usepackage{amsmath}
\usepackage{mathtools}
\usepackage{hyperref}
\usepackage{algorithm}
\usepackage{enumitem}

\hypersetup{colorlinks=true,citecolor=blue,linkcolor=blue,urlcolor=black}

\newtheorem{remark}{Remark}[section]
\newtheorem{example}[remark]{Example}

\begin{document}
	\title{Eigenfunction Behavior and Adaptive Finite Element Approximations of Nonlinear Eigenvalue Problems in Quantum Physics
	\thanks{This work was partially supported by the National Science Foundation of China under grants 91730302 and 11671389 and the Key Research Program of Frontier Sciences of the Chinese Academy of Sciences under grant QYZDJ-SSW-SYS010.
	}}
	\author{Bin Yang\thanks{LSEC, Institute of Computational Mathematics and Scientific/Engineering Computing, Academy of Mathematics and Systems Science, Chinese Academy of Sciences, Beijing 100190, China; and School of Mathematical Sciences, University of Chinese Academy of Sciences, Beijing 100049, China (binyang@lsec.cc.ac.cn).}
	\and Aihui Zhou\thanks{LSEC, Institute of Computational Mathematics and Scientific/Engineering Computing, Academy of Mathematics and Systems Science, Chinese Academy of Sciences, Beijing 100190, China; and School of Mathematical Sciences, University of Chinese Academy of Sciences, Beijing 100049, China ({azhou@lsec.cc.ac.cn}).}}
	\date{}
	\maketitle
	\begin{abstract}
		In this paper, we investigate a class of nonlinear eigenvalue problems resulting from quantum physics.  We first prove that the eigenfunction cannot be a polynomial on any open set, which may be reviewed as a refinement of the classic unique continuation property. Then we apply the non-polynomial behavior of eigenfunction to show that the adaptive finite element approximations are convergent even if the initial mesh is not fine enough. We finally remark that similar arguments can be applied to a class of linear eigenvalue problems that improve the relevant existing result.
	\end{abstract}
	\vskip 0.2cm

	
	\section{Introduction}
	
	In this paper, we investigate the eigenfunction behavior and adaptive finite element approximations of the following nonlinear eigenvalue problem:
	find $(\lambda_i,\phi_i)\in\mathbb{R}\times H_0^1(\Omega)~(i=1,2,\ldots,N)$ such that
	\begin{equation}\label{Orig-Eq}
		\left\{\begin{aligned}
		\left(-\kappa\Delta+V+\mathcal{N}(\rho) \right)\phi_i&=\lambda_i \phi_i,\quad \text{in~} \Omega,\quad i = 1,2,\ldots,N,\\
		\int_{\Omega}\phi_i \phi_j&=\delta_{ij},
		\end{aligned}
		\right.
	\end{equation}
	where $\Omega\subset\mathbb{R}^3$, $N\ge 1$, $\kappa>0$, $V:\Omega\to \mathbb{R}$ is a given function,  $\rho=\sum\limits_{i=1}^N|\phi_i|^2$, and $\mathcal{N}$ maps a nonnegative function to some function on $\Omega$.
 	We observe that  Schr\"{o}dinger-Newton equation modeling the quantum state reduction \cite{harrison03,penrose96}, Gross-Pitaevskii equation (GPE) describing  Bose-Einstein condensates (BEC) \cite{bao13,zhou04} and  Thomas-Fermi-von Weizs\"{a}cker (TFvW) type equations and Kohn-Sham equations appearing in electronic density functional theory \cite{becke14,chen13,bris03,lieb81,martin04} are typical examples of \eqref{Orig-Eq}.

	We understand that it is significant to solve eigenvalue problem \eqref{Orig-Eq} accurately and efficiently. And we note that the a priori knowledge of their eigenfunctions is very helpful in designing and analysis of numerical methods. To improve the approximation accuracy and reduce the computational cost in solving the eigenvalue problem, we see from the regularity of eigenfunction \cite{gong-shen-zhang-zhou08,zhang-zhou16} that adaptive finite element approaches should be employed (see also \cite{chen14,chen11-afem,dai-he-zhou15,dai-xu-zhou08,davydov-young-steinmann16,morin09,motamarri12,tsuchida-tsukada96} and references cited therein). We observe that the adaptive finite element analysis of nonlinear eigenvalue problem \eqref{Orig-Eq} in \cite{chen14,chen11-afem,chen11} requires that the initial mesh size is small enough. However, our numerical experiments show that the small initial mesh size requirement is unnecessary \cite{chen14,chen11-afem,chen11}. In this paper,  we study the adaptive finite element approximations when the initial mesh is not fine, for which we need to apply an eigenfunction behavior that is also investigated.

	We see that the unique continuation property is significant in the context of partial differential equations (see, e.g., \cite{jerison85,reed78,wolff93} and references cited therein). After looking into the behavior of eigenfunction of \eqref{Orig-Eq}, we find that the eigenfunction cannot be a polynomial on any open subset, which may be reviewed as a refinement of the classic unique continuation property and is indeed a key in our adaptive finite element analysis. Taking into account the eigenfunction behavior,  we are indeed able to prove the convergence of adaptive finite element approximations without the requirement of small initial mesh size.
	
	The rest of this paper is organized as follows. In the next section, we describe some basic notation and review the adaptive finite element method for solving eigenvalue problem \eqref{Orig-Eq}. Then we show some polynomial property which is crucial in our adaptive finite element analysis. In Section \ref{sec:behaviour-eigenf},
	we obtain that any eigenfunction of problem \eqref{Orig-Eq} cannot be a polynomial on any open subset of $\Omega$ under some assumptions, which may be reviewed as an extension and refinement of the classic unique continuation property. In Section \ref{sec:afem}, based on the non-polynomial behavior of eigenfunctions, we study the convergence of the adaptive finite element method. We finally remark that similar arguments can be applied to a class of linear eigenvalue problems  that improve the relevant existing result.
	
	\section{Preliminaries}
	Let $\Omega\subset\mathbb{R}^3$ be a polyhedral bounded domain. Let $\mathbb{Q}_+=\mathbb{Q}\cap[0,\infty)$ and $\alpha=(\alpha_1,\alpha_2,\alpha_3)\in\mathbb{Q}_+^3$ is an $3$-tuple. We denote $|\alpha|=\alpha_1+\alpha_2+\alpha_3$, define $x^\alpha=\xi_1^{\alpha_1}\xi_2^{\alpha_2}\xi_3^{\alpha_3}$ for any $x=(\xi_1,\xi_2,\xi_3)\in\mathbb{R}^3$, and use notation
	\[
	\partial_i \phi=\frac{\partial\phi}{\partial \xi_i}, ~i=1,2,3.
	\]
	For any $\alpha=(\alpha_1,\alpha_2,\alpha_3),\beta=(\beta_1,\beta_2,\beta_3)\in\mathbb{Q}^3_+$, we denote $\alpha\succ \beta$ if the first non-zero element of $\alpha-\beta=(\alpha_1-\beta_1,\alpha_2-\beta_2,\alpha_3-\beta_3)$ is greater than $0$ and $\alpha\succcurlyeq \beta$ if $\alpha\succ\beta$ or $\alpha=\beta$. For convenience, we define
	\[
	\mathcal{P}_{\mathbb{Q}_+}(\Omega) = \left\{ \sum_{\alpha\in I\subset\mathbb{Q}_+^3}  a_\alpha x^\alpha: x\in\Omega,~|I|<\infty \text{ and }a_\alpha \in\mathbb{R},\forall \alpha\in I \right\},
	\]
	where $|I|$ means the cardinality of $I$.
	We shall use the notation
	\[
	\mathcal{P}_{\mathbb{Q}_+}^{\mu}(\mathbb{R}^3)=\left\{ p^\mu: p\in\mathcal{P}_{\mathbb{Q}_+}(\mathbb{R}^3)\text{ and }p^{\mu} \text{ is the real function} \right\}
	\]
	for any $\mu\ge 0$.
	We call $a_\alpha x^\alpha$ with $a_\alpha\ne 0$ a monomial. Denote $|\alpha|$ the degree of monomial $a_\alpha x^\alpha$. We shall let the degree of polynomial $0$ be $-\infty$. For any $p\in \mathcal{P}_{\mathbb{Q}_+}(\Omega)$, define $\deg p$ as the max degree of terms of $p$, which is called the degree of $p$. We shall also denote $\deg p^{\mu}=\mu\deg p$ for any $p\in\mathcal{P}_{\mathbb{Q}_+}(\mathbb{R}^3)$ and $\mu\in\mathbb{Q}_+$ and $\deg(p/q) = \deg p - \deg q$ for any $p,q\in\mathcal{P}_{\mathbb{Q}_+}(\mathbb{R}^3)$. Let $\mathcal{P}_\ell(\Omega)$ be the set of real polynomials on $\Omega$ with degrees not greater than $\ell$. It is clear that $\mathcal{P}_\ell(\Omega)\subset\mathcal{P}_{\mathbb{Q}_+}(\Omega)$.
	The standard notation $H^s(\Omega)(s\ge 0)$ for Sobolev space and their associated norms $\|\cdot\|_{s,\Omega}$ shall also be used \cite{adams75}. We write
	\[
	G\subset\subset\Omega
	\]
	if $\bar{G}\subset\Omega$ and $\bar{G}$ is compact.
	$H^s_{\text{loc}}(\Omega)$ denotes the space of function $v$ satisfying that for any open set $G\subset\subset\Omega$, $v\in H^s(G)$. We use $\mathscr{P}(s,(c_1,c_2))$ to denote a class of functions satisfying some growth conditions:
	\begin{equation*}
	\mathscr{P}(s,(c_1,c_2))=\left\{ f:\exists ~a_1,a_2\in \mathbb{R} \mbox{ such that }
	c_1 t^{s}+a_1 \leq f(t) \leq c_2 t^{s}+a_2\quad \forall t\ge 0 \right\}
	\end{equation*}
	with $c_1\in \mathbb{R}$ and $c_2, s\in [0,\infty)$.
	
	\subsection{Quantum eigenvalue problem}
	We consider nonlinear eigenvalue problem \eqref{Orig-Eq} when $V$ has a form of
	\begin{equation}\label{V-nonlin-nonloc}
	V=-\sum_{j=1}^M\frac{f_j}{g_j}
	\end{equation}
	with $p_j,q_j\in\mathcal{P}_{\mathbb{Q}_+}^{\mu}(\mathbb{R}^3)(j=1,2,\cdots,M)$ for some $\mu\in\mathbb{Q}_+$, $\mathcal{N}$ is divided into two parts:
	\begin{equation}\label{N-nonlin-nonloc}
	\mathcal{N}(\rho)=\mathcal{N}_1(\rho)+\mathcal{N}_2(\rho),
	\end{equation}
	where $\rho=\displaystyle\sum_{i=1}^N|\phi_i|^2,~\mathcal{N}_1:[0,\infty)\to \mathbb{R}$ is defined by
	\begin{equation}\label{N1}
		\mathcal{N}_1(t) = \sum_{i=1}^K \frac{p_i(t)}{h_i(t)}\ln q_i(t)
	\end{equation}
	with $p_i,h_i\in\mathcal{P}_{\mathbb{Q}_+}(\mathbb{R})$, polynomials $q_i$ satisfying $q_i|_{[0,\infty)}>0$ and $q_i\ne 1(i=1,2,\cdots,K)$,
	and $\mathcal{N}_2$ is given by a convolution integral
	\begin{equation}\label{N2}
	\mathcal{N}_2(\rho)=\alpha\int_{\Omega}\frac{\rho(y)}{|\cdot-y|}\textup{d}y
	\end{equation}
	with some constant $\alpha$.
	
	The energy functional associated with \eqref{Orig-Eq} is
	\begin{equation*}\label{eq:energy}
	E(\Phi)=\int_\Omega\left(\kappa\sum_{i=1}^N|\nabla\phi_i|^2+V\rho_\Phi+\mathcal{E}(\rho_\Phi)\right)+\frac{\alpha}{2} D(\rho_\Phi,\rho_\Phi)
	\end{equation*}
	for $\Phi=(\phi_1,\phi_2,\ldots,\phi_N)\in\mathcal{H}\equiv (H_0^1(\Omega))^N$, where $\rho_\Phi=\displaystyle\sum_{i=1}^N|\phi_i|^2$, $\mathcal{E}:[0,\infty)\to \mathbb{R}$ is defined
	by
	\[
	\mathcal{E}(s)=\int_0^s\mathcal{N}_1(t)\textup{d}t,
	\]
	and $D(\cdot,\cdot)$ is a bilinear form as follows
	\begin{equation*}\label{eq:D(f,g)}
	D(f,g)=\int_\Omega\int_\Omega \frac{f(x)g(y)}{|x-y|}\textup{d}x\textup{d}y.
	\end{equation*}
	For any $\Phi\in\mathcal{H}$, we denote
	\begin{equation*}
		\|\Phi\|_{s,\Omega}=\left( \sum_{i=1}^N \|\phi_i\|_{s,\Omega} \right)^{1/2},\quad\|\Phi\|_{0,p,\Omega}=\left( \sum_{i=1}^N \|\phi_i\|_{0,p,\Omega} \right)^{1/2}.
	\end{equation*}
	
	We see that \eqref{Orig-Eq} includes the GPE, the Schr\"odinger-Newton equation, the TFvW type equation, and the Khon-Sham equation (see Remark \ref{rmk:GPE}, Example \ref{eg:SN}, Example \ref{eg:TFW}, and Example \ref{eg:KS} for more details).
	
	Let $\mathcal{Q}$ be a subspace of $\mathcal{H}$:
	\[
	\mathcal{Q}=\{\Phi\in\mathcal{H}:\Phi^T\Phi=I^{N\times N}\},
	\]
	where $\Phi^T\Psi=\left(\displaystyle\int_\Omega\phi_i\psi_j\right)\in\mathbb{R}^{N\times N}$. The ground state charge density of system \eqref{Orig-Eq} is obtained by solving minimization problem
	\begin{equation}\label{minimization-KS}
	\inf\{E(\Phi):\Phi\in\mathcal{Q}\}.
	\end{equation}
	We see that any minimizer $\Phi=(\phi_1,\ldots,\phi_N)$ of \eqref{minimization-KS} satisfies
	\begin{equation}\label{weakKS}
	\left\{
	\begin{aligned}
	\langle H_\Phi\phi_i,v \rangle&=\left(\sum_{j=1}^N\lambda_{ji}\phi_j,v\right)\quad\forall v\in H_0^1(\Omega),\quad i=1,2,\ldots,N,\\
	\int_\Omega\phi_i\phi_j&=\delta_{ij},
	\end{aligned}
	\right.
	\end{equation}
	where $H_\Phi:H_0^1(\Omega)\to H^{-1}(\Omega)$ is a Hamiltonian operator defined by
	\[
	\langle H_\Phi u,v \rangle =\kappa(\nabla u,\nabla v)+\left( V\rho_\Phi+\mathcal{N}(\rho_\Phi)u,v\right) \quad\forall u,v\in H_0^1(\Omega)
	\]
 	and
	\[
	\Lambda=(\lambda_{ij})_{i,j=1}^N=\left( \int_\Omega\phi_i H_\Phi\phi_j \right)_{i,j=1}^N
	\]
	is the Lagrange multiplier. We call $(\Lambda,\Phi)$ a ground state of \eqref{weakKS} and define the set of ground states by
	\[
	\Theta=\left\{ (\Lambda,\Phi)\in\mathbb{R}^{N\times N}\times\mathcal{Q}: E(\Phi)=\min_{\Psi\in \mathcal{Q}}E(\Psi)\text{ and } (\Lambda,\Phi)\text{ solves } \eqref{weakKS} \right\}.
	\]
	We  define the set of states of \eqref{weakKS} by
	\[
	\mathcal{W}=\{(\Lambda,\Phi)\in\mathbb{R}^{N\times N}\times\mathcal{H}:(\Lambda,\Phi)\text{ solves } \eqref{weakKS}\}.
	\]
	
	Since electron density $\rho_\Phi$ and operator $H_\Phi$ are invariant under any unitary transform, we may diagonalize Lagrange multipliers $\Lambda$ and arrive at
	\begin{equation}\label{eq:weakKS}
	\left\{
	\begin{aligned}
	\langle H_\Phi\phi_i,v \rangle&=\mu_i(\phi_i,v)\quad\forall v\in H_0^1(\Omega),\quad i=1,2,\ldots,N,\\
	\int_\Omega\phi_i\phi_j&=\delta_{ij},
	\end{aligned}
	\right.
	\end{equation}
	which is equivalent to \eqref{weakKS} and a weak form of \eqref{Orig-Eq}.
	
	\subsection{An adaptive finite element method}
	Let $d_\Omega$ be the diameter of $\Omega$ and $\{\mathcal{T}_h\}$ be a shape regular family of nested conforming meshes over $\Omega$ with size $h\in (0,d_\Omega)$: there exists a constant $\gamma^*$ such that
	\[
	\frac{h_\tau}{\rho_\tau}\le \gamma^*\quad \forall\tau\in \mathcal{T}_h,
	\]
	where $h_\tau$ is the diameter of $\tau$, $\rho_\tau$ is the diameter of the biggest ball contained in $\tau$, and $h=\max\{h_\tau : \tau\in\mathcal{T}_h\}$. Let
	$\mathcal{E}_h$ denote the set of interior faces of $\mathcal{T}_h$.
We shall also use a slightly abused of notation that $h$ denotes the mesh size function defined by
	\[
	h(x)=h_\tau, \quad x\in \tau \quad\forall \tau\in\mathcal{T}_h.
	\]
	Let $S^h(\Omega)\subset H^1(\Omega)$ be the corresponding finite element space consisting of continuous piecewise polynomials over $\mathcal{T}_h$ of degrees no greater than $n\ge 1$ and
	\[
	S_0^h(\Omega)=S^h\cap H_0^1(\Omega).
	\]
	Let $V_h=(S_0^h(\Omega))^N$.
	
Consider the finite element approximation of \eqref{minimization-KS}:
	\begin{equation}\label{d-min-KS}
	\inf\left\{ E(\Phi_h):\Phi_h\in V_h\cap\mathcal{Q} \right\}.
	\end{equation}
	We see that any minimizer $\Phi_h=(\phi_{1,h},\ldots,\phi_{N,h})$ of \eqref{d-min-KS} solves  Euler-Lagrange equation
	\begin{equation}\label{eq:d-weakKS}
	\left\{
	\begin{aligned}
	\langle H_{\Phi_h} \phi_{i,h},v \rangle&=\left( \sum_{j=1}^N\lambda_{ji,h}\phi_{j,h},v \right)\quad\forall v\in S_0^h(\Omega),\quad i=1,2,\ldots,N,\\
	\int_\Omega\phi_{i,h}\phi_{j,h}&=\delta_{ij}
	\end{aligned}
	\right.
	\end{equation}
	with  Lagrange multiplier
	\[
	\Lambda_h=(\lambda_{ij,h})_{i,j=1}^N=\left( \int_\Omega\phi_{i,h}H_{\Phi_h}\phi_{j,h} \right)_{i,j=1}^N
	\]
	when the energy functional is differentiable. Define the set of finite dimensional ground state solutions:
	\[
	\Theta_h=\left\{ (\Lambda_h,\Phi_h)\in\mathbb{R}^{N\times N}\times (\mathcal{Q}\cap V_h ):E(\Phi_h)=\min_{\Psi\in \mathcal{Q}\cap V_h} E(\Psi),\, (\Lambda_h,\Phi_h)\text{ solves } \eqref{eq:d-weakKS} \right\}.
	\]
	With using the unitary transformation, we have the following discrete Kohn-Sham equation
	\begin{equation}\label{eq:weakDKS}
	\left\{
	\begin{aligned}
	\langle H_{\Phi_h}\phi_{i,h},v \rangle&=\mu_{i,h}(\phi_{i,h},v)\quad\forall v\in H_0^1(\Omega),\quad i=1,2,\ldots,N,\\
	\int_\Omega\phi_{i,h}\phi_{j,h}&=\delta_{ij}.
	\end{aligned}
	\right.
	\end{equation}
	
	We recall that the adaptive finite element method is to repeat the following procedure \cite{chen14}:
	\[
	\text{Solve}\to\text{Estimate}\to\text{Mark}\to\text{Refine}.
	\]
	For convenience, we shall replace  subscript $h$ (or $h_k$) by an iteration counter $k$ of the adaptive method afterwards.
	
	Given an initial triangulation $\mathcal{T}_0$ so that the dimension of $S_0^h$ is greater than or equal to $N$. The above procedure generates a sequence of nested triangulations $\mathcal{T}_k(k=1,2,\cdots)$. Given an iteration counter $k$, procedure ``Solve" is to get the discrete solution over $\mathcal{T}_k$. Procedure ``Estimate" determines the element indicators for all elements $\tau\in\mathcal{T}_k$. In this step, a posteriori error estimators play an critical role. Then, element indicators are used by  procedure ``Mark" to create a subset $\mathcal{M}_k$ of marked elements $\tau\in\mathcal{T}_k$. To maintain mesh conformity, we usually partition a few more elements $\tau\in\mathcal{T}_k\setminus\mathcal{M}_k$ in  procedure ``Refine".
	
	Given a triangulation $\mathcal{T}_h$ and the corresponding finite element solution $(\Lambda_h,\Phi_h)$, we define finite element residual $\mathcal{R}_\tau(\Phi_h)$ and  jump $J_e(\Phi_h)$ by
	\begin{gather*}
	\mathcal{R}_\tau(\Phi_h)=\left( H_{\Phi_h}\phi_{i,h}-\sum_{j=1}^N\lambda_{ji}\phi_{j,h} \right)_{i=1}^N\quad \text{in }\tau\in\mathcal{T}_h,\\
	J_e(\Phi_h)=\left(j_e(\phi_{i,h})\right)_{i=1}^N,\quad j_e(\phi_{i,h})=\kappa\nabla\phi_{i,h}|_{\tau_1}\cdot\overrightarrow{n_1}+\kappa\nabla\phi_{i,h}|_{\tau_2}\cdot\overrightarrow{n_2},
	\end{gather*}
	where $e$ is the common face of elements $\tau_1$ and $\tau_2$ with unit outward normals $\overrightarrow{n_1}$ and $\overrightarrow{n_2}$, respectively. For $\tau\in\mathcal{T}_h$, we define the local error indicator $\eta_h(\Phi_h,\tau)$ as follows:
	\[
	\eta_h^2(\Phi_h,\tau)=h_\tau^2\|\mathcal{R}_\tau(\Phi_h)\|_{0,\Omega}^2+\sum_{e\in\mathcal{E}_h,e\subset\partial\tau}h_e\|J_e(\Phi_h)\|_{0,e}^2.
	\]
	Depending on the a posteriori error indicators $\{ \eta_k(\Phi_k,\tau) \}_{\tau\in \mathcal{T}_k}$,  procedure ``Mark" gives a strategy to create a subset of elements $\mathcal{M}_k$ of $\mathcal{T}_k$. Here, we consider ``maximum strategy" which only requires that the set of marked elements $\mathcal{M}_k$ contains at least one element of $\mathcal{T}_k$ holding the largest value estimator. Namely, there exists at least one element $\tau_k^\text{max}\in\mathcal{M}_k$ such that
	\begin{equation*}
	\eta_k(\Phi_k,\tau_k^\text{max})=\max_{\tau\in\mathcal{T}_k}\eta_k(\Phi_k,\tau).
	\end{equation*}
	
	The adaptive finite element algorithm for solving \eqref{eq:weakKS} is stated as follows \cite{chen14,chen11-afem,chen11}:
	\begin{algorithm}[H]
		\caption{}
		\label{algo:AFEM-KS}
		\begin{enumerate}
			\item Pick an initial mesh $\mathcal{T}_0$ and let $k=0$.
			\item Solve \eqref{eq:weakDKS} on $\mathcal{T}_k$ to get discrete solutions $(\mu_{i,k},\phi_{i,k})(i=1,2,\ldots,N)$.
			\item Compute local error indicates $\eta_k(\Phi_k,\tau)$ for all $\tau\in\mathcal{T}_k$.
			\item Construct $\mathcal{M}_k\subset\mathcal{T}_k$ by Maximum strategy.
			\item Refine $\mathcal{T}_k$ to get a new conforming mesh $\mathcal{T}_{k+1}$.
			\item Let $k=k+1$ and go to 2.
		\end{enumerate}
	\end{algorithm}

	We observe that there are a number of works on analyzing adaptive finite element methods in literature. We refer to \cite{bonito-demlow16,dai-he-zhou15,dai-xu-zhou08,gallistl15,giani-graham09} and references cited therein for linear eigenvalue problems and to \cite{chen14,chen11-afem,chen11} for nonlinear cases when the initial mesh is fine enough. We see that under the so-called Non-Degeneracy assumption\footnote{No eigenfunction is equal to a polynomial of degree $\le n$ on an open subset of $\Omega$, where $n$ denotes the polynomial degree of the finite element bases being used.}, \cite{morin09} proved convergence of an adaptive finite element method starting from any initial mesh for some linear elliptic eigenvalue problem.
	
	\subsection{A polynomial theory}
	In our analysis, we need the following basic results, which are motivated by \cite{zhou12}.
	\begin{lemma}\label{delta-sum-t^j}
		Let $k$ be a prime number and
		\[
		\delta=a_1 t+a_2 t^2+\cdots+a_{k-1} t^{k-1},
		\]
		where $t,a_i\in\mathbb{R}~(i=1,2,\ldots,k-1)$. Then there exist real polynomials
		\[
		\left\{p_j(t_1,t_2,\ldots,t_k): j=2,\ldots,k\right\}
		\]
		such that $p_j(t_1,t_2,\ldots,t_k)$ is a polynomial of degree $j-1$ with respect to $t_k$ and
		\begin{eqnarray*}\label{homogeneous-(k-1)}
			&p_j(\lambda t_1,\lambda t_2,\ldots,\lambda t_{k-1},t_k)=\lambda^j p_j(t_1,t_2,\ldots,t_{k-1},t_k) ~\forall \lambda\in \mathbb{R},\\
			&\delta^k+\sum_{j=2}^k p_j(a_1,a_2,\ldots,a_{k-1},t^k)\delta^{k-j}=0.
		\end{eqnarray*}
	\end{lemma}

	The proof of Lemma \ref{delta-sum-t^j} is given in Appendix \ref{appx:A}.
	\begin{lemma}\label{homogeneous}
		Suppose $k$ is a prime. Then for any positive integer $n$, there exist polynomials
		\[
		\left\{H_{n,j}(t_1,t_2,\ldots,t_n): j=0,1,2,\ldots,k^{n-1}\right\}
		\]
		with real coefficients satisfying
		\begin{enumerate}
			\item $H_{n,0}(t_1,t_2,\ldots,t_n)=1, ~H_{n,j}(t_1,t_2,\ldots,t_n)~(j=1,2,\ldots,k^{n-1})$ are homogeneous:
			\begin{align*}
			H_{n,j}(\lambda t_1,\lambda t_2,\ldots,\lambda t_n)=\lambda^j H_{n,j}(t_1,t_2,\ldots,t_n) ~\forall \lambda\in \mathbb{R},
			\end{align*}
			and $(-1)^{k^{n-1}}H_{n,k^{n-1}}(t_1,t_2,\ldots,t_n)$ is a monic polynomial of degree $k^{n-1}$ with respect to each variable $t_l(l=1,2,\ldots,n)$;
			
			\item if $\delta=\displaystyle\sum_{j=1}^n t_j$, then
			\[
			\sum_{j=0}^{k^{n-1}}H_{n,j}(t_1^k,t_2^k,\ldots,t_n^k)\delta^{k(k^{n-1}-j)}=0.
			\]
		\end{enumerate}
	\end{lemma}
	\begin{proof}
		We prove the conclusion by induction on $n$. Obviously, Lemma \ref{homogeneous} is true when $n=1$.
		Assume Lemma \ref{homogeneous} is true for $n\ge 1$. We show that Lemma \ref{homogeneous} is true for $n+1$. Let
		\[
		\sum_{j=1}^{n+1}t_j=\delta.
		\]
		It follows from the induction hypothesis and
		\[
		\sum_{j=1}^{n}t_j=\delta-t_{n+1}
		\]
		that there exist polynomials
		\[
		\{H_{n,j}(s_1,s_2,\ldots,s_n): j=0,1,2,\ldots,k^{n-1}\}
		\]
		with real coefficients satisfying that $H_{n,j}(s_1,s_2,\ldots,s_n)(j=0,1,2,\ldots,k^{n-1})$ are homogeneous, 
		$(-1)^{k^{n-1}} H_{n,k^{n-1}}(s_1,s_2,\ldots,s_n)$ is a monic polynomial of degree $k^{n-1}$ with respect to each variable $s_l(l=1,2,\cdots,n)$, and
		\[
		\sum_{j=0}^{k^{n-1}}H_{n,j}(t_1^k,t_2^k,\ldots,t_n^k)(\delta-t_{n+1})^{k(k^{n-1}-j)}=0.
		\]
		We obtain from Newton binomial theory that
		\[
		\sum_{i=1}^{k-1}a_i\delta^k\left(-\frac{t_{n+1}}{\delta}\right)^i=a,
		\]
		where
		\begin{align*}
		a&=-H_{n,k^{n-1}}(t_1^k,\ldots,t_n^k)-\delta^{k^n}-(-t_{n+1})^{k^n}-\sum_{\ell=1}^{k^{n-1}-1}\binom{k^n}{k\ell}\delta^{k^n-k\ell}(-t_{n+1})^{k\ell}\\
		&\quad -\sum_{j=1}^{k^{n-1}-1}H_{n,j}(t_1^k,\ldots,t_n^k)\sum_{\ell=0}^{k^{n-1}-j}\binom{k^n-kj}{k\ell}\delta^{k^n-kj-k\ell}(-t_{n+1})^{k\ell},\\
		a_i&=\sum_{j=0}^{k^{n-1}-1}H_{n,j}(t_1^k,\ldots,t_n^k)\sum_{\ell=1}^{k^{n-1}-j}\binom{k^n-kj}{k\ell-k+i}\delta^{k^n-kj-k\ell}(-t_{n+1})^{k\ell-k},\\
		&\quad i=1,2,\ldots,k-1.
		\end{align*}
		Since $k$ is a prime, there exist $\{p_j(t_1,t_2,\ldots,t_{k}): j=2,3,\ldots,k\}$ satisfying Lemma \ref{delta-sum-t^j}, namely,
		\begin{equation*}
		0=a^k+\sum_{j=2}^kp_j(a_1\delta^k,a_2\delta^k,\ldots,a_{k-1}\delta^k,(-t_{n+1})^k/\delta^k)a^{k-j},
		\end{equation*}
		or
		\begin{equation*}
		0=a^k+\sum_{j=2}^k p_j(a_1,a_2,\ldots,a_{k-1},(-t_{n+1})^k/\delta^k)\delta^{kj} a^{k-j}.
		\end{equation*}
		We conclude that Lemma \ref{homogeneous} is true when $n$ is replaced by $n+1$. This completes the proof.
	\end{proof}
	
	Since every integer greater than $1$ can be written as a product of one or more primes, we arrive at
	\begin{proposition}\label{k-th}
		 Let $k$ and $n$ be two positive integers. Then there exists a homogeneous polynomial $P(t_1,t_2,\ldots,t_{n+1})$ with real coefficients satisfying
		 \begin{enumerate}
		 	\item the degree of $P$ with respect to each variable is the same, and $P$ is a monic polynomial with respect to $t_{n+1}$;
		 	\item if $\delta=\displaystyle\sum_{j=1}^n t_j$, then
		 	\[
		 	P(t_1^k,t_2^k,\ldots,t_n^k,\delta^k)=0.
		 	\]
		 \end{enumerate}
	\end{proposition}
	
	We mention that the coefficients of the polynomial in Proposition \ref{k-th} can be integers and there exists a real homogeneous polynomial $P$ such that any zero of
		\[
		\sum_{i=1}^n\sqrt[k]{t_i}=\delta
		\]
		is an zero of
		\[
		P(t_1,t_2,\ldots,t_n,\delta^k)=0.
		\]
	
	\begin{proposition}\label{prop:pnonzero]}
		If $p\in\mathcal{P}_{\mathbb{Q}_+}(\mathbb{R}^3)$ and $\deg p > 0$, then for any open set $G\subset\mathbb{R}^3$, there exists $x_0\in G$ such that $p(x_0)\ne 0$.
	\end{proposition}
	\begin{proof}
	We see from the definition of $\mathcal{P}_{\mathbb{Q}_+}(\mathbb{R}^3)$ that
		\[
		p = \sum_{i=1}^n a_{\alpha^{(i)}} x^{\alpha^{(i)}},
		\]
		where $\alpha^{(i)}\in\mathbb{Q}_+^3$, $a_{\alpha^{(i)}}\in\mathbb{R}$, and $\alpha^{(n)}$ is the max index. Hence we can choose positive integer $k$ such that all components of $k\alpha^{(i)}(i=1,2,\ldots,n)$ are integers.
		
		Assume $p = 0$ in $G$. Then there exists a homogeneous polynomial $P(t_1,t_2,\ldots,t_n)$ satisfying the conclusion of Proposition \ref{k-th} and
		\[
		P(x^{k\alpha^{(1)}}, x^{k\alpha^{(2)}},\ldots, x^{k\alpha^{(n)}}) = 0,\quad\text{in }G.
		\]
		Set $Q = P(x^{k\alpha^{(1)}}, x^{k\alpha^{(2)}},\ldots, x^{k\alpha^{(n)}})$. Then $Q$ is a polynomial with positive which is a contradiction to $Q = 0$. This completes the proof.
	\end{proof}
	
	\begin{lemma}\label{lem:pisumlnsum}
		Let $n$ be an positive integer and $G$ be an open subset of $\mathbb{R}^3$. Let $p_j\in\mathcal{P}_{\mathbb{Q}_+}^{1/k}(\mathbb{R}^3)$ for some positive integer $k$ and $q_j\in\mathcal{P}_{\mathbb{Q}_+}(\mathbb{R}^3)(j=1,2,\ldots,n)$. If $q_j > 0(j=1,2,\ldots,n)$ in $G$ with $q_{j_0}\neq 1$ in $G$ for some $j_0$ and
		\[
		\deg p_{j_0} > \deg p_j \quad \forall j\in\{1,2,\ldots,n\}\setminus\{j_0\},
		\]
		then there exists $x_0\in G$ such that
		\[
		\sum_{j=1}^n p_{j}(x_0)\ln q_j(x_0)\ne 0.
		\]
	\end{lemma}

	The proof of Lemma \ref{lem:pisumlnsum} is provided in Appendix \ref{appx:B}.
	
	
	\section{Behavior of eigenfunction}\label{sec:behaviour-eigenf}
	In this section, we investigate the non-polynomial behavior of eigenfunctions of \eqref{Orig-Eq}, which will be applied to analyze convergence of their so-called adaptive finite element approximations. We may refer to \cite{gong-shen-zhang-zhou08,zhang-zhou16} for the regularity behavior of eigenfunctions that indeed result in applying adaptive finite element computations.
	
	We first recall the unique continuation property.
	
	\begin{definition}
		Equation \eqref{Orig-Eq} has a unique continuation property if every solution in $H^2_{\text{loc}}(\Omega)$ that vanishes on an open set of $\Omega$ vanishes identically.
	\end{definition}

	To look into if \eqref{Orig-Eq} has a unique continuation property, we may apply the following conclusion, which can be found in \cite{wolff93}.
	\begin{lemma}\label{UCP}
		Assume $u\in H_{\text{loc}}^2(\Omega)$ and $W\in L_{\text{loc}}^{3/2}(\Omega)$ such that $|\Delta u|\le W|u|$. If $u$ vanishes on an open set of $\Omega$, then $u$ is identically zero on $\Omega$.
	\end{lemma}

	\begin{theorem}\label{UCP-NonLin-Loc-NonLoc}
		If $V\in L^2(\Omega)$ and $\mathcal{N}_1(t)\in\mathscr{P}(s,(c_1,c_2)$ with $s\in[0,3/2]$, then \eqref{Orig-Eq} has a unique continuation property.
	\end{theorem}
	
	\begin{proof}
		It follows from \cite{chen09,dauge88} that $\phi_i\in H^2(\Omega)$, which together with Sobolev imbedding theorem leads to $\phi_i\in C({\bar \Omega})~(i=1,2,\ldots,N)$.
		
	 Not that Young's inequality and Sobolev imbedding theorem imply
	\[
		\|\mathcal{N}_2(\rho)\|_{0,\infty,\Omega}\le C\|\rho\|_{0,\Omega}\le C\sum_{i=1}^N \|\phi_i\|_{0,4,\Omega}^2\le C\|\Phi\|_{1,\Omega}^2<\infty.
	\]
	We have that $|\Delta \phi_i|=W_i|\phi_i|$	and $W_i\in L_{\text{loc}}^{3/2}(\Omega)$, where $W_i=|V+\mathcal{N}(\rho)-\lambda_i|/\kappa~(i=1,2,\ldots,N)$. Thus we arrive at the conclusion by using Lemma \ref{UCP}.
	\end{proof}
	
	\begin{remark}
		We may see from the proof of Theorem \ref{UCP-NonLin-Loc-NonLoc} that if $V\in L^2(\Omega)$ is replaced by $V\in L^{3/2}_{\textup{loc}}(\Omega)$ and any solution of $\eqref{Orig-Eq}$ is in $H^2(\Omega)$, then $\eqref{Orig-Eq}$ has a unique continuation property.
	\end{remark}
	
	\begin{theorem}\label{thm:nonpoly-K=0}
		Let $V$ and $\mathcal{N}$ be defined by \eqref{V-nonlin-nonloc} and \eqref{N-nonlin-nonloc}-\eqref{N2} with $\alpha=0$, respectively. If $V$ is a non-constant function and
		{\small
		\begin{equation}\label{eq:degmax}
		\deg p_1 - \deg h_1>\displaystyle\max\left\{0,~\max_{1\le j\le M}(\deg f_j-\deg g_j)/2,~\max_{2\le i\le K}(\deg p_i-\deg h_i)\right\},
		\end{equation}
		}then for any solution $\Phi=(\phi_1,\phi_2,\ldots,\phi_N)$ of \eqref{Orig-Eq}, there exists an eigenfunction $\phi_j(j\in\{1,2,\cdots,N\})$ being not a non-zero polynomial on any open set $G\subset \Omega$. If in addition, $V\in L^2(\Omega)$ and $\mathcal{N}_1(t)\in\mathscr{P}(s,(c_1,c_2)$ with $s\in[0,3/2]$, then there exists an eigenfunction  $\phi_j(j\in\{1,2,\cdots,N\})$ being not the polynomial on any open set $G\subset \Omega$.
	\end{theorem}
	\begin{proof}
		Assume that all eigenfunctions $\{\phi_1,\phi_2,\ldots,\phi_N\}$ are polynomials on some open set $G$: $\phi_j\in \mathcal{P}_\ell(G)(j=1,2,\cdots,N)$ for some positive integer $\ell$. Without loss of generality, let $\deg\phi_1\ge \displaystyle\max_{2\le j\le N}\deg\phi_j$. We have $\deg\rho = 2\deg\phi_1$ and see from \eqref{Orig-Eq} that
		\begin{equation}\label{eq:alpha0sum0}
		-\kappa\Delta \phi_1-\sum_{j=1}^M\frac{f_j}{g_j}\phi_1+\sum_{i=1}^K\frac{p_i(\rho)}{h_i(\rho)}\phi_1\ln q_i(\rho)=\lambda_1 \phi_1,~\text{in }G.
		\end{equation}
		If $\deg\phi_1>0$, then we see from \eqref{eq:degmax} that
		\begin{align*}
			(\deg p_1-\deg h_1)\deg\rho+\deg\phi_1 &>\deg \Delta\phi_1,\\
			(\deg p_1-\deg h_1)\deg\rho+\deg\phi_1 &>\deg\left(\frac{f_j}{g_j}\phi_1\right),~ j=1,2,\ldots,N,\\
			(\deg p_1-\deg h_1)\deg\rho+\deg\phi_1 &>(\deg p_i-\deg h_i)\deg\rho+\deg\phi_1,~ i=2,3,\ldots,K,\\
			(\deg p_1-\deg h_1)\deg\rho+\deg\phi_1 &>\deg \phi_1.
		\end{align*}
		Since $q_i$ are polynomials implying $q_i(\rho)\in\mathcal{P}_{\mathbb{Q}_+}(\Omega)(i=1,2,\cdots,K)$, we obtain from Lemma \ref{lem:pisumlnsum} that
		\[
		-\kappa\Delta \phi_1(x_0)-\sum_{j=1}^M\frac{f_j(x_0)}{g_j(x_0)}\phi_1(x_0)+\sum_{i=1}^K\frac{p_i(\rho)(x_0)}{h_i(\rho)(x_0)}\phi_1(x_0)\ln q_i(\rho)(x_0)\ne\lambda_1 \phi_1(x_0)
		\]
		for some $x_0\in G$, which is a contradiction to \eqref{eq:alpha0sum0}. Thus
		we arrive at that $\deg \phi_1=0$ on $G$. Since $\deg \phi_1 \ge \displaystyle\max_{2\le j\le N}\deg\phi_j$, we have that $\phi_j=c_j(j=1,2,\cdots,N)$  are constants on $G$. If $c_j\neq 0$ for all $j\in\{1,2,\ldots,N\}$, then
		\[
		\sum_{j=1}^M\frac{f_j}{g_j}=\sum_{i=1}^K\frac{p_i(\rho)}{h_i(\rho)}\ln q_i(\rho)-\lambda_1,\quad\text{in } G,
		\]
		with constant $\rho=\displaystyle\sum_{j=1}^N c_j^2$, which is impossible. Hence $c_j=0$ for some $j\in\{1,2,\ldots,N\}$.
		
		If in addition, $V\in L^2(\Omega)$ and $\mathcal{N}_1(t)\in\mathscr{P}(s,(c_1,c_2)$ with $s\in[0,3/2]$, then Theorem \ref{UCP-NonLin-Loc-NonLoc} implies that $\phi_j=0$ in $\Omega$ for some $j\in \{1,2,\ldots,N\}$, which is a contradiction to $\displaystyle\int_\Omega \phi_j^2=1$. This completes the proof.
	\end{proof}
	
	\begin{remark}\label{rmk:GPE}
		Note that Theorem \ref{thm:nonpoly-K=0} may be also true even if
		\[
		\deg p_1-\deg h_1=\max\limits_{1\le j\le M}(\deg f_j-\deg g_j)/2.
		\]
		For instance, no eigenfunction $\phi\in H^2(\Omega)$ of GPE \cite{bao13,zhou04}
		\[
		\left(-\frac12\Delta+V+\beta|\phi|^2\right)\phi=\lambda \phi
		\]
		with a harmonic trap potential
		\[
		V(x)=\gamma_1 \xi_1^2+\gamma _2 \xi_2^2+\gamma_3 \xi_3^2,~\gamma_1,\gamma_2,\gamma_3>0
		\]
		can be a polynomial on any open set $G\subset\Omega$, where $x=(\xi_1,\xi_2,\xi_3)\in \mathbb{R}^3$.
	\end{remark}
	
	
	\begin{theorem}\label{Thm-tfw}
		Let $V$ and $\mathcal{N}$ be defined by \eqref{V-nonlin-nonloc} and \eqref{N-nonlin-nonloc}-\eqref{N2} with $\alpha\ne 0$, respectively. Suppose $\alpha\Delta V$ is not a positive constant function. If  either
		\[
		\displaystyle\max_{1\le i\le K}(\deg p_i-\deg h_i)\le 1 ~\mbox{and}~\max\limits_{1\le j\le M}(\deg f_j-\deg g_j)<4,
		\]
		or $\deg q_1=0$ and
		{\small
		\begin{equation}\label{neq:maxdeg}
		\deg p_1 - \deg h_1>\displaystyle\max\left\{2,~\max_{1\le j\le M}(\deg f_j-\deg g_j)/2,~\max_{2\le i\le K}(\deg p_i-\deg h_i)\right\},
		\end{equation}
		}then for any solution $\Phi=(\phi_1,\phi_2,\ldots,\phi_N)$ of \eqref{Orig-Eq}, there exists an eigenfunction $\phi_j(j\in \{1,2,\cdots,N\})$ being not the non-zero polynomial on any open set $G\subset \Omega$.
		If in addition, $V\in L^2(\Omega)$ and $\mathcal{N}_1(t)\in\mathscr{P}(s,(c_1,c_2)$ with $s\in[0,3/2]$, then there exists an eigenfunction $\phi_j(j\in \{1,2,\cdots,N\})$ being not the polynomial on any open set $G\subset \Omega$.
	\end{theorem}
	
	\begin{proof}
		Assume that all eigenfunctions $\{\phi_1,\phi_2,\ldots,\phi_N\}$ are  polynomials on some open set $G$: $\phi_j\in \mathcal{P}_\ell(G)~(j=1,2,\cdots,N)$ for some positive integer $\ell$. Without loss of generality, let $\deg\phi_1\ge \displaystyle\max_{2\le j\le N}\deg\phi_j$. Obviously, $\deg\rho = 2\deg\phi_1$.

		If $\phi_1(x)\ne 0$ for any $x\in G$ and $\deg\phi_1>0$, then we obtain from \eqref{Orig-Eq} that
		\[
		-\kappa\frac{\Delta \phi_1}{\phi_1}-\sum_{j=1}^M\frac{f_j}{g_j}+\alpha\int_{\Omega}\frac{\rho(y)}{|\cdot-y|}\textup{d}y+\sum_{i=1}^K\frac{p_i(\rho)}{h_i(\rho)}\ln q_i(\rho)=\lambda_1,\quad\text{in }G.
		\]
		Applying Laplace operator to both sides yields
		\begin{equation}\label{LaplaceAgain}
			-\kappa p_{\phi_1}-\sum_{j=1}^M f_{j,g}-4\alpha\pi\rho+\sum_{i=1}^K p_{i,h,\rho}\ln q_i(\rho)+\sum_{i=2}^K p_{i,h,q,\rho}=0,\quad\text{in }G,
		\end{equation}
		where
		\begin{align*}
		p_{\phi_1}&=\Delta\left(\frac{\Delta \phi_1}{\phi_1}\right),\\
		f_{j,g}&=\Delta\left(\frac{f_j}{g_j}\right),~j=1,\ldots,M,\\
		p_{i,h,\rho} &= \Delta\left(\frac{p_i(\rho)}{h_i(\rho)}\right),~i=1,\ldots,K,\\
		p_{i,h,q,\rho} &= 2\frac{\nabla(p_i(\rho)/h_i(\rho))\cdot\nabla q_i(\rho)}{q_i(\rho)}+\frac{p_i(q_i(\rho)\Delta q_i(\rho)-|\nabla q(\rho)|^2)}{h_i(\rho)q_i^2(\rho)},~i=1,\ldots,K.
		\end{align*}
		
		If $\displaystyle\max_{1\le i\le K}(\deg p_i-\deg h_i)\le 1$  and $\max\limits_{1\le j\le M}(\deg f_j-\deg g_j)<4$, then only $4\alpha\pi\rho$ has the max degree.
		
		If $\deg q_1=0$ and \eqref{neq:maxdeg} holds, then $p_{1,h,q,\rho}=0$. Thus only one term, which is one term of $p_{1,h,\rho}$, has the max degree.
		
		Therefore, we get a contradiction to \eqref{LaplaceAgain} from Lemma \ref{lem:pisumlnsum}. Consequently, $\deg \phi_1\ge \displaystyle\max_{2\le j\le N}\deg \phi_j$ leads to that $\phi_j=c_j(j=1,2,\cdots,N)$ are constants in $G$.
		
		If $c_j\not=0$ for all $j\in\{1,2,\cdots,N\}$, we then derive from \eqref{LaplaceAgain} that
		\[
		\Delta V=4\alpha\pi \sum_{i=1}^N c_i^2 \quad \forall x\in G,
		\]
		which is impossible. Hence  $c_j=0$ for some $j\in\{1,2,\cdots,N\}$.
		
		If in addition, $V\in L^2(\Omega)$ and $\mathcal{N}_1(t)\in\mathscr{P}(s,(c_1,c_2)$ with $s\in[0,3/2]$, then we complete the proof by using Theorem \ref{UCP-NonLin-Loc-NonLoc}.
	\end{proof}

	We may apply Theorem \ref{thm:nonpoly-K=0} or Theorem \ref{Thm-tfw} to typical mathematical models in quantum physics to see the eigenfunction behavior.
	\begin{example}\label{eg:SN}
		No eigenfunction of  Schr\"odinger-Newton equation \cite{harrison03}
		\begin{align*}
			\left(-\Delta -\int_{\Omega}\frac{|u(y)|^2}{|\cdot-y|}\textup{d}y\right) u=\lambda u,~\text{in }\mathbb{R}^3
		\end{align*}
 		can be a polynomial on any open set of $\mathbb{R}^3$.
	\end{example}
	
	\begin{example}\label{eg:TFW}
		No eigenfunction of
		Thomas-Fermi-Dirac-von Weiz\"{s}acker equation \cite{chen11-afem,lieb81}
		\begin{equation*}\label{eq:TFDW}
		\left(-\kappa\Delta-\sum_{j=1}^M\frac{Z_j}{|\cdot-r_j|}+\int_{\Omega}\frac{\rho(y)}{|\cdot-y|}\textup{d}y+\beta_1 u^{2\nu-2}-\beta_2 u^{2/3}\right)u=\lambda u
		\end{equation*}
		can be a polynomial locally for an rational number $\nu$ in $[1,2]$, where $\beta_1$ and $\beta_2$ are constants.
	\end{example}
	
	\begin{example}\label{eg:KS}
		Kohn-Sham equation of a system consisting of $M$ nuclei of charges $\{Z_1,Z_2,\ldots,Z_M\}$ located at the positions $\{r_1,r_2,\ldots,r_M\}$ and $N$ electrons is as follows:
		{\small \begin{equation}\label{KS}
			\left\{\begin{aligned}
			\left(-\frac12\Delta+V_\text{\rm ext}+\int_{\Omega}\frac{\rho(y)}{|\cdot-y|}\textup{d}y+V_\text{\rm xc}(\rho)\right)\phi_i&=\lambda_i \phi_i,~ \text{in~} \Omega,~ i = 1,2,\ldots,N,\\
			\int_{\Omega}\phi_i \phi_j&=\delta_{ij},
			\end{aligned}
			\right.
		\end{equation}
		}where $V_\text{\rm ext}=-\sum\limits_{k=1}^M\frac{Z_k}{|\cdot-r_k|}$ is the associated external potential, $\rho=\sum\limits_{i=1}^N|\phi_i|^2$ is the electronic density, and $V_{\text{\rm xc}}(\rho)$ is the exchange-correlation potential such as $X_\alpha$ exchange-correction potential \cite{slater51}
		\begin{equation}\label{eq:xcXalpha}
			V_{\textup{xc}}(\rho)=\frac{3}{2}\alpha\left(\frac{3}{\pi}\rho\right)^{1/3}
		\end{equation}
		with $\alpha\in [2/3,1]$ and Perdew-Zunger type local-density approximations (LDA) potential \cite{perdew-zunger81}:
		{\small
		\begin{equation}\label{eq:xcLDA}
			V_{\text{xc}}^{\text{LDA}}(\rho)=\left\{
			\begin{aligned}
			&-\frac{0.1423+0.0633 r_{s}+0.1748 \sqrt{r_{s}}}{(1+1.0529 \sqrt{r_{s}}+0.3334 r_{s})^{2}}-\left(\frac{9}{4\pi^2}\right)^{\frac13}\frac{1}{r_s},&&\text{if~}r_s\ge 1, \\
			&0.0311 \ln r_{s}-0.0584+0.0013 r_{s} \ln r_{s}-0.0084 r_{s}-\left(\frac{9}{4\pi^2}\right)^{\frac13}\frac{1}{r_s},&&\text{if~}r_s<1
			\end{aligned}
			\right.
		\end{equation}
		}with $r_s=(\frac{3}{4\pi\rho})^{1/3}$.
		
		We see that if the exchange-correction potential is chosen as either \eqref{eq:xcXalpha} or \eqref{eq:xcLDA}, then for any solution  $\Phi=(\phi_1,\phi_2,\ldots,\phi_N)$ of \eqref{KS}, there exists an eigenfunction $\phi_j~(j\in \{1,2,\cdots,N\})$ being not the polynomial on any open set $G\subset \Omega$.
%
		In fact, the same conclusion is true for Vosko-Wilk-Nusair type LDA \cite{vwm80}
		\begin{equation*}
			\begin{aligned}
			V_{\textup{xc}}(\rho)=\frac{A}{2}\{\ln \frac{t^{2}}{X(t)}+\frac{2 b}{Q} \tan ^{-1} \frac{Q}{2 t+b}&-\frac{b t_{0}}{X(t_{0})}(\ln \frac{\left(t-t_{0}\right)^{2}}{X(t)}\\+& \frac{2\left(b+2 t_{0}\right)}{Q} \tan ^{-1} \frac{Q}{2 t+b} ) \},
			\end{aligned}
		\end{equation*}
		where $r_s=(\frac{3}{4\pi\rho})^{1/3}$, $t=r_s^{1/2}$, $X(t)=t^{2}+b t+c$, $Q=\left(4 c-b^{2}\right)^{1 / 2}$, $A=0.0621814$, $t_0 =-0.409286$, $b=13.0720$, and $c=42.7198$.
	\end{example}

	Indeed, we conjecture that no eigenfunction of \eqref{Orig-Eq} can be a polynomial on any open set in $\mathbb{R}^3$ when $N>1$. Unfortunately, it is still open whether it is true or not.
	
	\section{Adaptive approximations}\label{sec:afem}
	In this section, we apply the behavior of the eigenfunctions to investigate the convergence of adaptive finite element approximations of \eqref{Orig-Eq}. We assume that
	\begin{enumerate}[label=(\roman*)]
		\item $V\in L^2(\Omega)$;
		\item $\mathcal{E}\in\mathscr{P}(3,(c_1,c_2))$ with $c_1\ge 0$ or $\mathscr{P}(4/3,(c_1,c_2))$;
		\item $\mathcal{N}_1\in\mathscr{P}(s_1,(c_1,c_2))$ for some $s_1\in [0,2)$ and $t\mathcal{N}'(t)\in\mathscr{P}(s_2,(\tilde{c}_1,\tilde{c}_2))$ for some $s_2\in [0,2)$.
	\end{enumerate}
	
	Let
	\[
	\Omega^+=\bigcup_{T\in\mathcal{T}^+}T,
	\]
	where \begin{equation*}
	\mathcal{T}^+=\bigcup_{k\ge 0}\bigcap_{m\ge k}\mathcal{T}_m.
	\end{equation*}
	
	In our analysis, we need Lemma 4.3 in \cite{morin09}, which is stated as follows:
	\begin{lemma}\label{lem:equi-hk-0}
		The set $\Omega^+$ is empty if and only if $\displaystyle\lim_{k\to\infty}\|h_k\|_{0,\infty,\Omega}=0$.
	\end{lemma}
	
	We observe from Theorem 4.2 in \cite{chen11-afem} and  Theorem 3.5 in \cite{chen14} that approximations $\Theta_k$ produced by Algorithm \ref{algo:AFEM-KS} may converge to a solution of \eqref{Orig-Eq} for any initial mesh and the solution becomes a ground state if the initial mesh size is sufficiently small so that $\Theta_0$ is sufficiently near to $\Theta$. Indeed, based on the eigenfunction behavior, we are able to prove that $\Theta_k$ produced by Algorithm \ref{algo:AFEM-KS} may converge to a ground state of \eqref{Orig-Eq} starting from any initial mesh.
	
	Using the similar argument to the proof of Lemma 6.2 in \cite{morin09}, we have
	\begin{lemma}\label{lem:meshsize-0-NL}
		Let $\{ h_k \}_{k\in\mathbb{N}}$ and $\{ \Theta_k \}_{k\in\mathbb{N}}=\{ (\Lambda_k,\Phi_k) \}_{k\in\mathbb{N}}$ be produced by Algorithm \ref{algo:AFEM-KS}. If there exists an eigenfunction of \eqref{Orig-Eq} being not a polynomial on any open set $G\subset\Omega$, then $\|h_k\|_{0,\infty,\Omega}\to 0$ as $k\to \infty$.
	\end{lemma}
	\begin{proof}
		We obtain from the proof of Theorem 3.5 in \cite{chen14} that there exists a subsequence $\Phi_{k_m}$ and some solution $\Phi$ of \eqref{Orig-Eq} such that $\Phi_{k_m}\to\Phi$ in $\mathcal{H}$. Without loss of generality, we assume that $\phi_1$ cannot be a polynomial on any open set $G\subset \Omega$. As a result, $\phi_{1,k_m}\to\phi_1$ in $H^1_0(\Omega)$ as $m\to\infty$.
		
		If $\|h_k\|_{0,\infty,\Omega}$ does not tend to zero, then we derive from  Lemma \ref{lem:equi-hk-0} that $\Omega^+$ is not empty. Thus there exist $T\in\mathcal{T}^+$ and $k_0\in\mathbb{N}$ such that $T\in\mathcal{T}_k$ for all $k\ge k_0$. Since $\displaystyle\lim_{m\to\infty}\|\phi_{1,k_m}-\phi_1\|_{0,T}=0$ and $\phi_{1,k}|_T\in\mathcal{P}_n(T)$ for some integer $n$, we obtain from that $\mathcal{P}_n(T)$ is a finite dimensional space that $\phi_1\in\mathcal{P}_n(T)$, which contradicts to that $\phi_1$ cannot be a polynomial on any open set $G\subset\Omega$. This completes the proof.
	\end{proof}
	

	Combining Theorem \ref{thm:nonpoly-K=0}, Theorem \ref{Thm-tfw}, and Lemma \ref{lem:meshsize-0-NL}, we obtain that mesh size $h_k$ tends to zero under the assumption in Theorem \ref{thm:nonpoly-K=0} or that in Theorem \ref{Thm-tfw}, which means that the mesh size will be sufficiently small after finite iteration steps. Namely, approximate set $\Theta_k$ is sufficiently close to $\Theta$ provided $k\gg 1$. Let the distance between sets $X,Y\subset\mathbb{R}^{N\times N}\times\mathcal{H}$ be defined by
	\[
	d_\mathcal{H}(X,Y)=\sup_{(\Lambda,\Phi)\in X}\inf_{(\Sigma,\Psi)\in Y} (|\Lambda-\Sigma|+\|\Phi-\Psi\|_{1,\Omega}),
	\]
	where $|\cdot|$ is the  Frobenius norm in $\mathbb{R}^{N\times N}$. Due to existing works \cite{chen14} and \cite{chen11-afem}, we arrive at
	
	\begin{theorem}
		Let $\{\Theta_k\}_{k\in\mathbb{N}}$ be the sequence generated by Algorithm {\rm\ref{algo:AFEM-KS}}. If
		\[
		\min_{\Psi\in\mathcal{Q}} E(\Psi)<\inf_{(M,\Psi)\in\mathcal{W}\setminus\Theta} E(\Psi)
		\]
		and the assumption in Theorem \ref{thm:nonpoly-K=0} or that in Theorem \ref{Thm-tfw} is satisfied, then
		\begin{align*}
		&\lim_{k\to\infty} E_k=\min_{\Psi\in \mathcal{Q}} E(\Psi),\\
		&\lim_{k\to\infty} d_\mathcal{H}(\Theta_k,\Theta)=0,
		\end{align*}
		where $E_k=E(\Phi)~((\Lambda,\Phi)\in\Theta_k)$.
	\end{theorem}

	As a result, we see from \cite{chen14,chen11} that the adaptive finite element method has asymptotic linear convergence rate and asymptotic optimal complexity from any initial mesh. More precisely, the adaptive finite element method has linear convergence rate and optimal complexity after finite iteration steps.
	
	\section{Concluding remarks}
	In this paper, we have investigated a class of nonlinear eigenvalue problems modeling quantum physics.  We have first proved that the eigenfunction cannot be a polynomial on any open set, which may be reviewed as a refinement of the standard unique continuation property. Then applying non-polynomial behavior of the eigenfunctions, we have shown that  adaptive finite element approximations are convergent even if the initial mesh is not fine enough.
	
	We mention that the same conclusion can be expected for any dimensions greater than  $3$. For instance, our arguments can be applied to the following linear eigenvalue problem:
	\begin{equation}\label{eq:potential}
		-\nabla\cdot(\mathcal{A}\nabla u)+\mathcal{V}u=\lambda \mathcal{B}u,\quad\text{in }\Omega.
	\end{equation}
	where $\Omega\subset\mathbb{R}^d$ for positive integer $d\ge 3$ and $\mathcal{A}$ is a symmetric-matrix-valued function and is uniformly positive definite. We see that \eqref{eq:potential} includes   electronic Schr{\" o}dinger equation
	\begin{equation}\label{schrodinger}
		\left(-\sum_{i=1}^N\frac{\hbar^2}{2m_e}\nabla^2_{x_i}-\sum_{i=1}^N\sum_{j=1}^{M}\frac{Z_je^2}{|x_i-r_j|}+\frac{1}{2}\sum_{i,j=1, i\ne j}^N\frac{e^2}{|x_i-x_j|}\right)\phi=E\phi,\quad\text{in }\mathbb{R}^{3N},
	\end{equation}
	where $\hbar$ is Planck's constant divided by $2\pi$, $m_e$ is the mass of the electron, $\{x_i:i=1,\cdots,N\}$ are the variables that 	describe the electron positions, and $e$ is the electronic charge, $\phi$ is the wavefunction, $N$ is the number of electrons, $M$ is the number of atoms, $Z_j$ is the atomic number of the $j$-th atom, and $r_j$ is the position of the $j$-th atom.

	For convenience of discussion, we introduce the following assumptions:
	\begin{enumerate}[label=\Roman*]
		\item Entries of $\mathcal{A}$ are continuous and piecewise functions in $\mathcal{P}_{\mathbb{Q}_+}(\Omega)$.
		\item  $\mathcal{B}$ is a piecewise function in $\mathcal{P}_{\mathbb{Q}_+}(\Omega)$.
		\item $\mathcal{V}=-\displaystyle\sum_{j=1}^M \frac{f_j}{g_j}$, where $f_j,~q_j~(j=1,2,\ldots,M)$ are piecewise functions in $\mathcal{P}_{\mathbb{Q}_+}^{\mu}(\Omega)$ for some $\mu\in\mathbb{Q}_+$.
		\item $\mathcal{V}$ cannot be equal to $\lambda\mathcal{B}$ for any $\lambda\in\mathbb{R}$ in any open subset of $\Omega$.
	\end{enumerate}
 	\vskip 0.2cm

	If Assumptions I-IV hold true and that entries of $\mathcal{A}|_G, \mathcal{B}|_G$ belong to $ \mathcal{P}_{\mathbb{Q}_+}$ and $f_j|_G,g_j|_G\in\mathcal{P}_{\mathbb{Q}_+}^{\mu}$ for all $j\in\{1,2,\ldots,M\}$ imply that only one among $\deg\mathcal{A}|_G-2,$ $\deg\mathcal{B}|_G,\deg f_1|_G-\deg g_1|_G,\ldots,\deg f_M|_G-\deg g_M|_G$ equals to
	\[
	\max\{\deg\mathcal{A}|_G-2,\deg f_1|_G-\deg g_1|_G,\ldots,\deg f_M|_G-\deg g_M|_G,\deg\mathcal{B}|_G\}
	\]
	when $G\subset\Omega$ is an open subset, then  no eigenfunction of \eqref{eq:potential} can be a non-zero polynomial on $G$. If in addition, \eqref{eq:potential} has a unique continuation property (see, e.g., \cite{schechter80,wolff93}), then any $H^2_{\text{loc}}$ eigenfunction of \eqref{eq:potential} cannot be a polynomial on any open subset of $\Omega$. Since \eqref{schrodinger} satisfies  Assumptions I-IV, in particular, we obtain more sophisticate conclusion than that in the existing literature (see, e.g., \cite{reed78}).

Note that the so-called Non-Degeneracy Assumption of a linear case of
	\begin{equation}\label{eq:linearP}
		-\nabla\cdot(\mathcal{A}\nabla u)=\lambda \mathcal{B}u
	\end{equation}
	has been introduced in \cite{morin09}, which is a special case of \eqref{eq:potential} when $\mathcal{V}=0$, entries of $\mathcal{A}$ are continuous and piecewise linear, and $\mathcal{B}$ is piecewise constant,  with which together convergence of an adaptive finite element method from any initial mesh for \eqref{eq:linearP} is then derived.

	\appendix
	\section*{}\label{appx:A}
	In this appendix, we provide a proof of Lemma \ref{delta-sum-t^j}, whose idea is inspired by Appendixes A and B of \cite{pesic04}.
	\begin{proof}
 		We may view $\delta$ as a polynomial  $\delta(t)$ with respect to $t$. Let $z\ne 1$ be a $k$th root of $1$. We have $\displaystyle\sum_{j=0}^{k-1}z^j=0$ and
		\begin{equation*}\label{z^h}
		\left\{z^{m j}: j=1,2,\ldots,k-1\right\}=\left\{z^j: j=1,2,\ldots,k-1\right\}
		\end{equation*}
		for any positive integer $m$ that is not divisible by $k$. Let
		\begin{equation}\label{ori_pty}
		P_t(y)=\prod_{m=0}^{k-1}\left(y-\delta(z^m t)\right),\quad y\in\mathbb{R},
		\end{equation}
		we obtain that $P_t(\delta)=0$. We claim that $P_t(\delta)=0$ yields the conclusion.
		
		Indeed, it follows from \eqref{ori_pty} that $P_t(y)$ can be rewritten as
		\begin{equation}\label{expand_pty}
		P_t(y)=y^k+\sum_{j=1}^k q_j(a_1,a_2,\ldots,a_{k-1},t)  y^{k-j},
		\end{equation}
		where
		\begin{equation}\label{pty_qj}
		q_j(t_1,t_2,\ldots,t_k)=\sum_{\ell=j}^{j(k-1)}f_{j,\ell}(t_1,t_2\ldots,t_{k-1})t_k^\ell
		\end{equation}
 		and  $f_{j,\ell}(\ell=j,j+1,\ldots,j(k-1))$ are homogeneous polynomials of degree $j$. We see from \eqref{ori_pty} that $\overline{P_t(y)}=P_t(y)$ and $q_j~(j=1,2,\ldots,k)$ are real polynomials. We obtain from \eqref{ori_pty}-\eqref{pty_qj} that
		\begin{align*}
		P_t(y)
		&=\frac{1}{k}\sum_{m=0}^{k-1} P_{z^m t}(y)=y^k+\sum_{j=1}^k\left(\frac1k\sum_{m=0}^{k-1} q_j(a_1,a_2,\ldots,a_{k-1},z^m t)\right)y^{k-j}\\
		&=y^k+\sum_{j=1}^k\left(\frac1k\sum_{\ell=j}^{j(k-1)}f_{j,\ell}(a_1,a_2\ldots,a_{k-1})t^\ell\sum_{m=0}^{k-1} z^{\ell m}\right)y^{k-j}\\
		&=y^k+\sum_{j=2}^k\left(\sum_{m=1}^{j-1}f_{j,mk}(a_1,a_2,\ldots,a_{k-1})t^{mk}\right)y^{k-j},
		\end{align*}
		namely,
		\begin{equation}\label{simplify_pty}
		P_t(y)=y^k+\sum_{j=2}^k\left(\sum_{m=1}^{j-1}f_{j,mk}(a_1,a_2,\ldots,a_{k-1})t^{mk}\right)y^{k-j}.
		\end{equation}
		Comparing \eqref{expand_pty} with \eqref{simplify_pty}, we arrive at
		\begin{align}
		q_1(t_1,t_2,\ldots,t_k)&=0,\notag\\
		q_j(t_1,t_2,\ldots,t_k)&=\sum_{m=1}^{j-1}f_{j,mk}(t_1,t_2,\ldots,t_{k-1})t_k^{mk},~j=2,\ldots,k.\label{simplify_qj}
		\end{align}
		Pick up $p_j(t_1,t_2,\ldots,t_{k})$ such that \[
		p_j(t_1,t_2,\ldots,t_{k-1},t_k^k)=q_j(t_1,t_2,\ldots,t_k)~(j=2,3,\ldots,k).
		\]
		We complete the proof by using that \eqref{simplify_qj} and $f_{j,\ell}$ are homogeneous of degree $j$ and  $P_t(\delta)=0$.
	\end{proof}

	\section*{}\label{appx:B}
	In this appendix, we provide a proof of Lemma \ref{lem:pisumlnsum}.
	\begin{proof}
		Without loss of generality, we divide $n$ into two parts: $n=n_1+n_2$, such that $\deg q_j = 0$ for $j\in\{0,1,\ldots,n_1\}$ and $\deg q_j>0$ for $j\in\{n_1+1,\ldots,n_1+n_2\}$.
		
		We prove the conclusion by induction on $n_2$.
		
		\noindent (1) For $n_2=0$, we prove the conclusion by contradiction again. Assume that
		\begin{equation*}
		\sum_{j=1}^{n_1} p_j \ln q_j=0,\quad\text{in }\tilde{G},
		\end{equation*}
		for some open subset $\tilde{G}\subset\Omega$, where $\ln q_j$ are constants and $\ln q_{j_0}(x)\neq 0$ for any $x\in \tilde{G}$.
		
		For convenience, we assume $j_0 = 1$.  Let $P(t_1,t_2,\ldots,t_{n_1})$ be a homogeneous polynomial satisfying the conclusion of Proposition \ref{k-th} and
		\begin{equation*}
		P(p_2^k,\ldots,p_{n_1}^k,p_1^k)=0,\quad\text{in }\tilde{G}.
		\end{equation*}
		Set $Q=P(p_2^k,\ldots,p_n^k,p_1^k)$. We get from $p_j\in\mathcal{P}_{\mathbb{Q}_+}^{1/k}(j=1,2,\ldots,n_1)$ that $Q\in\mathcal{P}_{\mathbb{Q}_+}$. Note that $P$ is a homogeneous polynomial and monic in $t_n$, we obtain from the definition of $Q$ and $\deg p_1>\displaystyle\max_{2\le j\le n_1}\deg p_j$ that $\deg Q > 0$. Therefore Proposition \ref{prop:pnonzero]} leads to a contradiction to $Q = 0$ in $\tilde{G}$. Thus Lemma \ref{lem:pisumlnsum} is true for $n_2=0$.
		
		\noindent (2) Assume Lemma \ref{lem:pisumlnsum} is true for $n_2\ge 0$. We show that Lemma \ref{lem:pisumlnsum} is true for $n_2+1$. Let $j_0=1$ or $j_0=n_1+1$. It is obvious that the conclusion is true if $p_{n_1+1} = 0$ in $G$. If $p_{n_1+1}\neq 0$ in $G$, then we may assume that
		\[
		\sum_{j=1}^{n_1}p_{j} \ln q_j+\sum_{j=n_1+1}^{n_1+n_2+1} p_{j} \ln q_j=0,\quad\text{in }G,
		\]
		which leads to
		\begin{equation}\label{eq:pisumlnsum0}
			\sum_{j=1}^{n_1} \frac{p_j}{p_{n_1+1}}\ln q_j + \ln q_{n_1+1} + \sum_{j=n_1+2}^{n_1+n_2+1} \frac{p_j}{p_{n_1+1}} \ln q_j=0,\quad\text{in }\tilde{G}
		\end{equation}
		for some open subset $\tilde{G}\subset G$, where $q_j(j=1,\ldots,n_1)$ are constants. Applying $\partial_i(i=1,2,3)$ to \eqref{eq:pisumlnsum0}, we obtain
		\begin{equation*}
			\sum_{j=1}^{n_1} \frac{\tilde{p}_j}{p^2_{n_1+1}}\ln q_j + \sum_{j=n_1+2}^{n_1+n_2+1}\frac{p_j\partial_i q_j}{p_{n_1+1}q_j} +\frac{\partial_i q_{n_1+1}}{q_{n_1+1}} + \sum_{j=n_1+2}^{n_1+n_2+1} \frac{\tilde{p}_j}{p^2_{n_1+1}} \ln q_j=0,\quad\text{in }\tilde{G},
		\end{equation*}
		where $\tilde{p}_j=p_{n_1+1}\partial_i p_j-p_j\partial_i p_{n_1+1}$, $j=1,\ldots,n_1+n_2+1$. It is easy to see that $p_{n_1+1}^k p_j^k\xi_i \tilde{p}_j\in \mathcal{P}_{\mathbb{Q}_+}^{1/k}$.
		
		If $j_0=1$, then there exists $i$ such that $\deg\tilde{p}_1=\deg p_1+\deg p_{n_1+1}-1$. Thus we have
		\begin{align*}
			\deg \tilde{p}_1 - \deg p^2_{n_1+1} &> \deg \tilde{p}_j - \deg p^2_{n_1+1},~ j=2,\ldots,n_1+n_2+1,\\
			\deg \tilde{p}_1 - \deg p^2_{n_1+1} &> \deg (p_j\partial_i q_j) - \deg(p_{n_1+1}q_j),~ j=n_1+2,\ldots,n_1+n_2+1,\\
			\deg \tilde{p}_1 - \deg p^2_{n_1+1} &> \deg\partial_i q_{n_1+1}-q_{n_1+1}.
		\end{align*}
		
		If $j_0 = n_1+1$, then we pick up $i$ satisfying $\deg_i q_{n_1+1} > 0$. It follows that
		$$\deg\partial_i q_{n_1+1} - \deg q_{n_1+1}=-1.$$
Consequently,
		\begin{align*}
			\deg\partial_i q_{n_1+1} - \deg q_{n_1+1} &> \deg \tilde{p}_j - \deg p^2_{n_1+1},~ j=1,\ldots,n_1+n_2+1,\\
			\deg\partial_i q_{n_1+1} - \deg q_{n_1+1} &> \deg (p_j\partial_i q_j) - \deg(p_{n_1+1}q_j),~ j=n_1+2,\ldots,n_1+n_2+1.
		\end{align*}
		Thus we conclude from the induction hypothesis that Lemma \ref{lem:pisumlnsum} is true when $n_2$ is replaced by $n_2+1$. This completes the proof.
	\end{proof}
	
\bibliographystyle{siam}

\end{document}